\documentclass[a4paper,oneside,11pt]{article} 

\usepackage[bottom]{footmisc}

\usepackage{amssymb}
\usepackage{tikz}
\usepackage{amsmath}

\usepackage{comment}

\usepackage{nameref}

\usepackage{bbm}

\usepackage{enumitem, hyperref}
\makeatletter
\def\namedlabel#1#2{\begingroup
    #2%
    \def\@currentlabel{#2}%
    \phantomsection\label{#1}\endgroup
}
\makeatother

\usepackage[all,cmtip]{xy}
\usepackage{tikz-cd}
\usepackage{graphicx}

\usepackage[all]{xy}

\usepackage{mathtools}
\DeclarePairedDelimiter\ev{\langle}{\rangle}

\usepackage[utf8]{inputenc} 
\usepackage[T1]{fontenc} 
\usepackage{textcomp} 
\usepackage[english]{babel}
\usepackage[autolanguage]{numprint} 
\usepackage{hyperref} 
\usepackage{graphicx} 
\usepackage{verbatim} 

\usepackage{amsmath,amssymb,amsthm} 
\usepackage{comment}

\usepackage{fancyhdr} 
\usepackage{lipsum}

\renewcommand\[{\begin{equation}}\renewcommand\]{\end{equation}} 
\renewcommand\epsilon\varepsilon 
\renewcommand\phi\varphi 

\newtheorem{innercustomconj}{Conjecture}
\newenvironment{customconj}[1]
  {\renewcommand\theinnercustomconj{#1}\innercustomconj}
  {\endinnercustomconj}

\newtheorem{innercustomthm}{Theorem}
\newenvironment{customthm}[1]
  {\renewcommand\theinnercustomthm{#1}\innercustomthm}
  {\endinnercustomthm}

\newcommand\NNN{\mathbb{N}} 
\newcommand\NN{\mathcal{N}}
\newcommand\TT{\mathcal{T}}
\newcommand\ZZ{\mathbb{Z}} 
\newcommand\ab\allowbreak 

\newcommand\GW{\operatorname{GW}}

\newcommand\VV{\mathcal{V}}

\newcommand\Spec{\operatorname{Spec}}

\newcommand\CHt{\widetilde{\operatorname{CH}}}

\newcommand\id{\operatorname{id}}

\newcommand\MW{\mathfrak{M}^{\operatorname{MW}}}

\newcommand\FFF{\mathcal{F}}

\newcommand\Sm{\operatorname{Sm}_k}

\newcommand\eeta{\boldsymbol{\eta}}
\newcommand\Om{\operatorname{\Omega}}
\newcommand\Ab{\mathcal{A}b}

\newcommand\KMW{\underline{\operatorname{K}}^{MW}}

\newcommand\AAA{\mathbb{A}}

\providecommand{\Codes}[1]
{
  \small	
  \textbf{\textit{MSC---}} #1
}

\newcommand\OO{\mathcal{O}}
\newcommand\LL{\mathcal{L}}

\newcommand\LLL{\omega}

\newcommand\SH{\mathbf{SH}}
\newcommand\PP{\mathbb{P}}

\newcommand\Grp{\mathbf{Grp}}

\theoremstyle{definition} 
\newtheorem{Def}{Definition}[subsection] 
\theoremstyle{plain} 
\newtheorem{Lem}[Def]{Lemma} 
\newtheorem{The}[Def]{Theorem} 
\theoremstyle{remark} 
\newtheorem{Exe}[Def]{Example} 
\newtheorem{Par}[Def]{} 

\title{A vanishing theorem for quadratic intersection multiplicities} 
\author{\sc Niels Feld} 

\date{} 

\begin{document} 

\maketitle 


\begin{abstract}
We study intersection theoretic problems in the setting of Chow-Witt groups with coefficients in a fixed Milnor-Witt cycle algebra over a perfect field. We prove that the product maps on such groups satisfy the following property: given two points in a regular local scheme with supports which do not intersect properly, their product vanishes. This gives an analogue of Serre's vanishing result for intersection multiplicities. 

\end{abstract}

\Codes{14C17, 14C35, 11E81}


\section{Introduction}

\subsection{Current work}

Classical intersection theory \cite{Fult} stands on the study of cycles in Chow groups. Thus, many classical results could be reinterpreted in the new context given by Chow-Witt groups. For example, if $R$ is a regular local ring of dimension $d$, and $M$ and $N$ are two $R$-modules of finite type such that the product $M\otimes N$ has a finite length, then Serre defines an intersection multiplicity 
\begin{center}
$\chi_R(M,N)=\sum_{i=0}^d (-1)^i \lg_R (\operatorname{Tor}_i^R(M,N))$,
\end{center}
where $\lg_R$ is the length of an $R$-module of finite type. Serre vanishing conjecture states that if $\dim M +\dim N  < d$, then $\chi_R(M,N)=0$ (see the work of Roberts \cite{Roberts85}, and Gillet-Soul\'e \cite{GilletSoule85,GilletSoule87} for a proof in the general case). The appeal of Serre's multiplicities comes from the fact that they can be used to compute the product of cycles in the Chow ring of a variety.
\par Following ideas of \cite{FaselSrin08}, we would like to have a similar description of the intersection product defined for Milnor-Witt cycle modules in \cite[Section 11]{Fel18}. In particular: does the intersection multiplicities of Serre have a quadratic interpretation? The question is difficult. Nevertheless, keeping in mind Serre vanishing conjecture and \cite[Conjecture 1]{FaselSrin08}, the following result seems plausible:
\begin{customconj}{1}
Let $(R,\mathfrak{m})$ a regular local ring of dimension $n$. Let $Z$ and $T$ be two closed subsets of $\Spec R$ such that $\dim Z + \dim T < n$ and $Z\cap T=\mathfrak{m}$. Then the multiplication \cite[Section 11]{Fel18}
\begin{center}
$\CHt^i_Z(R) \times \CHt^j_T(R) \to \CHt^{i+j}_{\mathfrak{m}}(R)$ 
\end{center}
is zero for any natural numbers $i,j\in \NNN$.
\end{customconj}

As evidence, we have the following theorem (see Subsection \ref{Notation} and Appendix \ref{Appendix_MW} for the notations and more details about the definitions).

\begin{customthm}{2}[see Theorem \ref{VanishingTheorem}] Let $M$ be a Milnor-Witt cycle algebra over a fixed perfect field $k$. Let $X$ be a regular local scheme of dimension $n$ over $k$ and denote by $x_0$ its closed point. Let $\VV_X$ be a virtual vector bundle over $X$. Let $Z$ and $T$ be closed subsets of $X$ such that $Z\cap T=\{ x_0 \}$. Then the intersection product
\begin{center}
$A^i_Z(X,M,\VV_Z) \times A^j_T(X,M,\VV_T) \to A^{i+j}_{x_0}(X,M,\VV_{x_0})$
\end{center}
is zero for any $i,j\in \ZZ$.

\end{customthm}
\par In particular for $M=\KMW$, this is true for the intersection product defined on the Chow-Witt ring of $X$. More generally, the theorem apply to any ring spectrum $M$ according to Theorem \ref{Thesis_Theorem}.
\par In the future, we hope to extend Theorem 2 to more general schemes $X/k$ (the theory of Milnor-Witt cycle modules can be defined over a large class of base scheme $S$, see \cite{DegliseFeldJin22}). Another direction would be to consider \textit{effective} MW-cycle modules in the sense of \cite{Fel20}.
\par Finally, we add that Conjecture 1 remains unclear if we do not assume the existence of a base field. Nevertheless, following the ideas of Gillet-Soulé \cite{GilletSoule85}, a proof may still be obtained by studying analogues of the Adams operations (see \cite{FaselHaution20}).

%

\subsection{Notation}\label{Notation}

Throughout the paper, we fix a perfect field $k$.
\par We denote by $\Grp$ and $\Ab$ the categories of (abelian) groups.
\par We consider only schemes that are essentially of finite type over $k$. All schemes and morphisms of schemes are defined over $k$. The category of smooth $k$-schemes of finite type is denoted by $\Sm$ and is endowed with the Nisnevich topology (thus, {\em sheaf} always means {\em sheaf for the Nisnevich topology}).
\par 
Let $X$ be a scheme and $x$ a point of $X$. We define the codimension of $x$ in $X$ to be $\dim (\OO_{X,x})$, the dimension of the localisation ring of $x$ in $X$ (see also \cite[TAG 02IZ]{stacks_project}). If $n$ a natural number, we denote by $X_{(n)}$ (resp. $X^{(n)}$) the set of point of dimension $n$ (resp. codimension $n$) of $X$ (this makes sense even if $X$ is not smooth).
\par By a field $E$ over $k$, we mean {\em a $k$-finitely generated field $E$}. Since $k$ is perfect, notice that $\Spec E$ is essentially smooth over $S$. We denote by $\FFF_k$ the category of such fields.
\par Let $f:X\to Y$ be a (quasi)projective lci morphism of schemes (e.g. a morphism between smooth schemes). Denote by $\LL_f$ (or $\LL_{X/Y}$) the virtual vector bundle over $Y$ associated with the cotangent complex of $f$ defined as follows: if $p:X\to Y$ is a smooth morphism, then $\LL_p$ is (isomorphic to) $\Om_{X/Y}$ the space of (Kähler) differentials. If $i:Z\to X$ is a regular closed immersion, then $\LL_i$ is the normal cone $-\NN_ZX$. If $f$ is the composite $\xymatrix{ Y \ar[r]^i & \PP^n_X \ar[r]^p & X}$  with $p$ and $i$ as previously (in other words, if $f$ is lci projective), then $\LL_f$ is isomorphic to the virtual tangent bundle $i^*\Omega_{\PP^n_X/X} - \NN_Y(\PP^n_X) $ (see also \cite[Section 9]{Fel18}).
Denote by $\LLL_f$ (or $\LLL_{X/Y}$) the determinant of $\LL_f$.
\par Let $X$ be a scheme and $x\in X$ a point, we denote by $\LL_{x}=(\mathfrak{m}_x/\mathfrak{m}_x^2)^{\vee}$ and $\LLL_{x/X}=\LLL_x$ its determinant. Similarly, let $v$ a discrete valuation on a field, we denote by $\LLL_{v}$ the line bundle $(\mathfrak{m}_v/\mathfrak{m}_v^2)^{\vee}$.
\par Let $E$ be a field (over $k$) and $v$ a valuation on $E$. We will always assume that $v$ is discrete. We denote by $\mathcal{O}_v$ its valuation ring, by $\mathfrak{m}_v$ its maximal ideal and by $\kappa(v)$ its residue class field. We consider only valuations of geometric type, that is we assume: $k\subset \mathcal{O}_v$, the residue field $\kappa(v)$ is finitely generated over $k$ and satisfies $\operatorname{tr.deg}_k(\kappa(v))+1=\operatorname{tr.deg}_k(E)$.

\par Let $E$ be a field. We denote by $\GW(E)$ the Grothendieck-Witt ring of symmetric bilinear forms on $E$. For any $a\in E^*$, we denote by $\ev{a}$ the class of the symmetric bilinear form on $E$ defined by $(X,Y)\mapsto aXY$ and, for any natural number $n$, we put $n_{\epsilon}=\sum_{i=1}^n \ev{-1}^{i-1}$. Recall that, if $n$ and $m$ are two natural numbers, then $(nm)_{\epsilon}=n_{\epsilon}m_{\epsilon}$.

\subsection*{Acknowledgement}
I deeply thank Frédéric Déglise, Jean Fasel, Marc Levine, Paul Arne \O stv\ae r, Bertrand To\" en, Fangzhou Jin and Baptiste Calmès. 
%


\section{Main theorem}

We fix $M$ a Milnor-Witt cycle algebra over the perfect base field $k$. We start with proving some geometric lemmas about the Chow-Witt groups with coefficients in $M$.
\subsection{Geometric lemmas}

\begin{Lem} \label{LemmIntersNulle}
Let $g:Y\to X$ be a smooth morphism of finite type schemes of constant fiber dimension $1$, $\sigma:X\to Y$ a section of $g$ and $\VV_X$ a virtual vector bundle over $X$. Let $i:Z\to X$ be a closed immersion and consider $\bar{Z}=g^{-1}(Z)$ the pullback
 along $g$. The induced map $\bar{\sigma}:Z\to \bar{Z}$ is such that the pushforward

\begin{center}
 $\bar{\sigma}_*:C_*(Z,M,\VV_Z)\to C_*(\bar{Z},M,\VV_{\bar{Z}})$
\end{center}
is zero on homology.
\end{Lem}
\begin{proof}
See \cite[Lemma 4.1.5]{Fel19}.
\end{proof}

\begin{Def}
If $k[\![ x'_1,\dots, x'_m]\!]\to k[\![ x_1,\dots, x_n]\!]$ is a morphism between power series algebras over $k$ induced by $x'_i\to \sum_j a_{ij}x_j$ for some $a_{ij}\in k$, we call the induced map 
$\Spec(k[\![ x_1,\dots, x_n]\!]) \to \Spec(k[\![ x'_1,\dots, x'_m]\!])$ a linear projection. Such linear projection are determined by points in $\AAA^{mn}(k)$.
\end{Def}

\begin{Lem}
\label{FasSriCor2.4}
Let $X=\Spec(k[\![ x_1,\dots, x_n]\!])$ and $Y=\Spec(k[\![ z_1,\dots, z_{n-1}]\!])$. Let $Z,T\subset X$ be closed subsets such that $\dim Z+ \dim T < \dim X$ and $Z\cap T$ is supported on the closed point. Then for any sufficiently  general\footnote{See \cite[Theorem 14.14]{matsumura1986} for a more precise definition of \guillemotleft sufficiently general\guillemotright.} linear projection $p:X\to Y$, we have:
\begin{itemize}
\item $Z\neq p^{-1}(p(Z))$
\item $p^{-1}(p(Z))\cap T$ is also supported on the closed point.
\end{itemize}
\end{Lem}
\begin{proof}
See \cite[Corollary 2.4]{FaselSrin08}.
\end{proof}

\begin{Lem}
\label{FasSriProp3.6}
Let $X=\Spec(k[\![ x_1,\dots, x_n]\!])$ and $Y=\Spec(k[\![ z_1,\dots, z_{n-1}]\!])$. Let $Z\subset X$ be a proper closed subset. Then for any integer $i$ and any sufficiently general linear projection $p:X\to Y$, then extension of support
\begin{center}
$A^i_Z(X,M,*)\to A^i_{p^{-1}(p(Z))}(X,M,*)$
\end{center}
is zero.
\end{Lem}
\begin{proof} (see also \cite{FaselSrin08})
As $Z$ is a proper closed subset of $X$, there exists a nonzero non-unit $t\in k[\![ x_1,\dots, x_n]\!]$ such that $Z\subset V(t)$. Let $j:V(t) \to X$ be the inclusion. Any sufficiently general linear projection $p:X\to Y$ is flat and has the property that $p_{|V(t)}:V(t)\to Y$ is finite. Consider the following fibre product:
\begin{center}
$ \xymatrix{
X' 
\ar[r]^f
\ar[d]_{p'}
&
X
\ar[d]^p
\\
V(t)
\ar[r]_{p_{|V(t)}}
\ar[ru]^j
&
Y.
}$
\end{center}
The inclusion $j:V(t) \to X$ induces a closed immersion $i':V(t) \to X')$ such that $fi'=j$. Observe that $V(t)$ is also a principal divisor in $X'$ (see \cite[Theorem 5.23]{Srinivas96}). As closed subsets, we have $p^{-1}(p(Z))=f(p'^{-1}(Z))$ and then it is enough to show that 
\begin{center}
$A^i_Z(X,M,*)\to A^i_{f(p'^{-1}(Z))}(X,M,*)$
\end{center}
is zero to get the result. The following diagram
\begin{center}
$\xymatrix{
A^i_Z(X,M,*)
\ar@{=}[r]
&
A^i_Z(X,M,*)
\ar[r]^-{\mathfrak{e}}
&
A^i_{f(p')^{-1}(Z)}(X,M,*)
\\
A^{i-1}_Z(V(t),M,*)
\ar[u]^{j_*}
\ar[r]_{(i')_*}
&
A^i_{i'(Z)}(X',M,*)
\ar[u]^{f_*}
\ar[r]_-{\mathfrak{e}}
&
A^i_{(p')^{-1}(Z)}(X',M,*)
\ar[u]^{f_*}
}$
\end{center}
is commutative, where $\mathfrak{e}$ is the extension of support. We can see that $j_*$ is an isomorphism. Therefore, Lemma \ref{LemmIntersNulle} shows that 
\begin{center}
$A^i_{i'(Z)}(X',M,*)\to A^i_{p'^{-1}(Z)}(X',M,*)$ 
\end{center}
is zero.
\end{proof}
\begin{Lem}
\label{InfiniteField}
Let $X$ be a scheme over $k$ and $g:X_{k(t)} \to X$ be the (smooth) base change. Then
\begin{center}
$g^*:A^*(X,M,*) \to A^*(X_{k(t)},M,-\LL_{X_{k(t)}/X}+*)$
\end{center}
is injective.
\end{Lem}
\begin{proof}
See \cite[Theorem 8.3]{Fel18}.
\end{proof}

\subsection{Proof of the main theorem}

\begin{The} \label{VanishingTheorem} Let $M$ be a Milnor-Witt cycle algebra. Let $X$ be a regular local scheme of dimension $n$ over $k$ and denote by $x_0$ its closed point. Let $V_X$ be a virtual vector bundle over $X$. Let $Z$ and $T$ be closed subsets of $X$ such that $Z\cap T=\{ x_0 \}$. Then the intersection product
\begin{center}
$A^i_Z(X,M,\VV_Z) \times A^j_T(X,M,\VV_T) \to A^{i+j}_{x_0}(X,M,\VV_{x_0})$
\end{center}
is zero for any $i,j\in \ZZ$.
\par In particular for $M=\KMW$, this is true for the intersection product defined on the Chow-Witt ring of $X$.

\end{The}

\begin{proof}
Let $\hat{X}$ be the completion of the local ring $X$ (for the ${x_0}$-adic valuation). By definition, we have $A^n_{x_0}(X,M,\VV_{x_0}) \simeq A^n_{\hat{x}_0}(\hat{X},M,\VV_{\hat{x}_0})$ for any integer $n$, and the following diagram
\begin{center}
$\xymatrix{
A^i_Z(X,M,\VV_Z)\times A^j_T(X,M,\VV_T)
\ar[r]
\ar[d]
&
A^{i+j}_{x_0}(X,M,\VV_{x_0})
\ar[d]
\\
A^i_{\hat{Z}}(\hat{X},M,\VV_{\hat{Z}})\times A^j_{\hat{T}}(\hat{X},M,\VV_{\hat{T}})
\ar[r]
&
A^{i+j}_{\hat{x}_0}(\hat{X},M, \VV_{\hat{x}_0})
}$
\end{center}
is commutative, where the vertical arrows are induced by the completion. Hence, it is enough to prove the result for a complete regular local scheme and we may assume that $X$ is the spectrum of the ring $A=k[\![ x_1,\dots, x_n]\!]$.
\par By Lemma \ref{InfiniteField}, we may also assume that $k$ is infinite. Now, put $B=k[\![ z_1,\dots, z_{n-1}]\!]$, and apply Lemma \ref{FasSriCor2.4} and \ref{FasSriProp3.6}: there exists a linear projection $p:X \to \Spec(B)$ such that:
\begin{enumerate}
\item The extension of support $\mathfrak{e}:A^i_Z(X)\to A^i_{p^{-1}(p(Z))}(X)$ is zero.
\item $p^{-1}(p(Z)) \cap T= x_0$.
\end{enumerate}
The conclusion follows from the following commutative diagram:
\begin{center}
$\xymatrix{
A^i_Z(X,M,\VV_Z)\times A^j_T(X,M,\VV_T)
\ar[r]
\ar[d]^{\mathfrak{e}\times \id}
&
A^{i+j}_{x_0}(X,M,\VV_{x_0})
\ar@{=}[d]
\\
A^i_{p^{-1}(p(Z))}(X,M,\VV_{p^{-1}(p(Z))})\times A^j_{{T}}(X,M,\VV_{T})
\ar[r]
&
A^{i+j}_{{x}_0}({X},M,\VV_{{x}_0}).
}$
\end{center}

\end{proof}

\appendix

\section{Recollection in motivic homotopy theory}
\label{Appendix_MW}
\subsection{Milnor-Witt cycle modules}

We denote by $\mathfrak{F}_k$ the category whose objects are the couple $(E, \mathcal{V}_E)$ where $E$ is a field over $k$ and $\mathcal{V}_E\in \mathfrak{V}(E)$ is a virtual vector space (of finite dimension over $F$). A morphism $(E,\mathcal{V}_E)\to (F, \mathcal{V}_F)$ is the data of a morphism $E\to F$ of fields over $k$ and an isomorphism $\mathcal{V}_E \otimes_E F \simeq \mathcal{V}_F$ of virtual $F$-vector spaces.
\par A morphism $(E,\mathcal{V}_E)\to (F, \mathcal{V}_F)$ in $\mathfrak{F}_k$ is said to be finite (resp. separable) if the field extension $F/E$ is finite (resp. separable).
\par We recall that a Milnor-Witt cycle modules $M$ over $k$ is a functor from $\mathfrak{F}_k$ to the category $\Ab$ of abelian groups  equipped with data 
\begin{description}
\item [\namedlabel{itm:D1}{D1}] (restriction maps) Let $\phi : (E,\mathcal{V}_E)\to (F, \mathcal{V}_F)$ be a morphism in $\mathfrak{F}_k$. The functor $M$ gives a morphism $\phi_*:M(E,\mathcal{V}_E) \to M(F,\mathcal{V}_F)$,
\item  [\namedlabel{itm:D2}{D2}] (corestriction maps) Let $\phi : (E,\mathcal{V}_E)\to (F, \mathcal{V}_F)$ be a morphism in $\mathfrak{F}_k$ where the morphism $E\to F$ is {\em finite}. There is a morphism  $\phi^*:M(F,\Om_{F/k}+\mathcal{V}_F) \to M(E,\Om_{E/k}+\mathcal{V}_E)$,
\item [\namedlabel{itm:D3}{D3}] (Milnor-Witt K-theory action) Let $(E,\mathcal{V}_E)$ and $(E,\mathcal{W}_E)$ be two objects of $\mathfrak{F}_k$. For any element $x$  of $\KMW(E,\mathcal{W}_E)$, there is a morphism 
\begin{center}
$\gamma_x : M(E,\mathcal{V}_E)\to M(E,\mathcal{W}_E+\mathcal{V}_E)$
\end{center}
so that the functor $M(E,-):\mathfrak{V}(E)\to \Ab$ is a left module over the lax monoidal functor $\KMW(E,-):\mathfrak{V}(E)\to \Ab$ (see \cite[Definition 39]{Yetter03}),
\item [\namedlabel{itm:D4}{D4}] (residue maps) Let $E$ be a field over $k$, let $v$ be a valuation on $E$ and let $\mathcal{V}$ be a virtual projective {$\mathcal{O}_v$-module} of finite type. Denote by $\mathcal{V}_E=\VV \otimes_{\mathcal{O}_v} E$ and $\VV_{\kappa(v)}=\VV \otimes_{\mathcal{O}_v} \kappa(v)$. There is a morphism
\begin{center}
$\partial_v : M(E,\VV_E) \to M(\kappa(v), - \NN_v+\VV_{\kappa(v)}),$
\end{center}
\end{description}   and satisfying compatibility rules (R1a),\dots, (R4a) .
Moreover, a Milnor-Witt cycle module $M$ satisfies axioms \namedlabel{itm:FD}{FD}  (finite support of divisors) and \namedlabel{itm:C}{C} (closedness) that enable us to define a complex $(C_p(X,M,\VV_X),d_p)_{p\in \ZZ}$ for any scheme $X$ and virtual bundle $\VV_X$ over $X$ where
\begin{center}

$C_p(X,M,\VV_X)=\bigoplus_{x\in X_{(p)}}
 M(\kappa(x),\Omega_{\kappa(x)/k}+\VV_x)$.
\end{center}

\begin{Exe}
The main example of MW-cycle module is given by Milnor-Witt K-theory $\KMW$ (see \cite{Fel18,Fel20bis} for more details).
\end{Exe}

\begin{Par} \label{FiveBasicMapsArticle2}
The complex $(C_p(X,M,\VV_X),d)_{p\geq 0}$ is called the {\em Milnor-Witt complex of cycles on $X$ with coefficients in $M$} and we denote by $A_p(X,M,\VV_X)$ the associated homology groups (called {\em Chow-Witt groups with coefficients in $M$}). We can define five basic maps on the complex level (see \cite[Section 4]{Fel18}):
\begin{description}
\item[Pushforward] Let $f:X\to Y$ be a $k$-morphism of schemes, let $\VV_Y$ be a virtual bundle over the scheme $Y$. The data \ref{itm:D2} induces a map
\begin{center}

$f_*:C_p(X,M,\VV_X)\to C_p(Y,M, \VV_Y)$.
\end{center}
\item[Pullback] Let $g:X\to Y$ be an essentially smooth morphism of schemes. Let $\VV_Y$ a virtual bundle over $Y$. Suppose $X$ connected (if $X$ is not connected, take the sum over each connected component) and denote by $s$ the relative dimension of $g$. The data \ref{itm:D1} induces a map
\begin{center}
$g^*:C_p(Y,M,\VV_Y) \to C_{p+s}(X,M,- \LL_{X/Y}+\VV_X)$.
\end{center}
 
 \item[Multiplication with units] Let $X$ be a scheme of finite type over $k$ with a virtual bundle $\VV_X$. Let $a_1,\dots, a_n$ be global units in $\mathcal{O}_X^*$. The data \ref{itm:D3} induces a map	
 \begin{center}
 $[a_1,\dots, a_n]:C_p(X,M,\VV_X) \to C_p(X,M,\ev{n}+\VV_X)$.
 \end{center}
 
 \item[Multiplication with $\eeta$]
 Let $X$ be a scheme of finite type over $k$ with a virtual bundle $\VV_X$. The Hopf map $\eeta$ and the data \ref{itm:D3} induces a map
 \begin{center}
 
 $\eeta:C_p(X,M,\VV_X)\to C_p(X,M,-\AAA^1_X+\VV_X)$.
 \end{center}
 
 \item[\namedlabel{itm:Bmaps}{Boundary map}]     
 Let $X$ be a scheme of finite type over $k$ with a virtual bundle $\VV_X$, let $i:Z\to X$ be a closed immersion and let $j:U=X\setminus Z \to X$ be the inclusion of the open complement. The data \ref{itm:D4} induces a map
 \begin{center}

 $\partial=\partial^U_Z:C_p(U,M,\VV_U) \to C_{p-1}(Z,M,\VV_Z)$.
 \end{center}

\end{description} 

 These maps satisfy the usual compatibility properties (see \cite[Section 5]{Fel18}). In particular, they induce maps $f_*,g^*, [u], \eeta, \partial^U_Z$ on the homology groups $A_*(X,M,*)$.
 
 \end{Par}

\begin{Def}

A pairing $M\times M'\to M''$ of MW-cycle modules over $k$ is given by bilinear maps for each $(E,\VV_E),(E,\mathcal{W}_E)$ in $\mathfrak{F}_k$
\begin{center}

$M(E,\VV_E)\times M'(E,\mathcal{W}_E)\to M''(E,\VV_E+\mathcal{W}_E)$
\\ $(\rho, \mu)\to \rho \cdot \mu$
\end{center}
which respect the $\KMW$-module structure and which are compatible with the data \textbf{D1}, \textbf{D2}, \textbf{D3} and \textbf{D4} in the sense of \cite[Definition 3.21]{Fel18}.
\par A ring structure on a MW-cycle module $M$ is a pairing 
\begin{center}
$M\times M\to M$
\end{center}
(in the sense of \cite[Definition 3.21]{Fel18}) which induces on 
\begin{center}
$\bigoplus_{\VV_E\in \mathfrak{V}(E)}M(E,\VV_E)$
\end{center}
an associative and $\epsilon$-commutative ring structure. In that case, we may say that $M$ is an \textit{algebra}.
\end{Def}
\begin{Exe}
By definition, a Milnor-Witt cycle module $M$ comes equipped with a pairing $\KMW \times M \to M$. When $M=\KMW$, this defines a ring structure on $M$.

\end{Exe}

\begin{Par}{\sc Cross products.} Let $M$ be a Milnor-Witt cycle module with a ring structure $M\times M \to M$ (see \cite[Definition 3.21]{Fel18}). Let $Y$ and $Z$ be two essentially smooth schemes over $k$ equipped with virtual vector bundles $\VV_Y$ and $\mathcal{W}_Z$. We define the cross product
\begin{center}

$\times : C_p(Y,M,\VV_Y)\times C_q(Z,M',\mathcal{W}_Z)\to C_{p+q}(Y\times Z,M'',\VV_{Y\times Z}+\mathcal{W}_{Y\times Z})$
\end{center}
as follows. For $y\in Y$, let $Z_y=\Spec \kappa(y) \times Z$, let $\pi_y:Z_y\to Z$ be the projection and let $i_y:Z_y\to Y\times Z$ be the inclusion. For $z\in Z$ we understand similarly $Y_z,\pi_z:Y_z\to Y$ and $i_z:Y_z\to Y\times Z$. We give the following two equivalent definitions:
\begin{center}
$\rho\times \mu = \displaystyle \sum_{y\in Y_{(p)}} (i_y)_*(\rho_y\cdot \pi_y^*(\mu))$,
\\ $\rho\times \mu = \displaystyle \sum_{z\in Z_{(q)}} (i_z)_*(\pi^*_z(\rho)\cdot \mu)$.
\end{center}
The cross product satisfies the expected properties (associativity, graded-commutativity, chain rule; see \cite[Section 11]{Fel18}).

\end{Par}

\begin{Par}

{\sc Intersection.} For $X$ smooth, the product induces a map
\begin{center}

$A^p(X,M,\VV_X)\times A^q(X,M,\mathcal{W}_X)\to A^{p+q}(X\times X,M,\VV_{X\times X}+\mathcal{W}_{X\times X})$.
\end{center}

By composing with the Gysin morphism 
\begin{center}

$\Delta^*:A^{p+q}(X\times X,M,\VV_{X\times X}+\mathcal{W}_{X\times X})\to A^{p+q}(X,M,-\TT_{\Delta}+\VV_X+\mathcal{W}_X)$
\end{center}
induced by the diagonal $\Delta:X\to X\times X$, we obtain the map
\begin{center}
$A^p(X,M,\VV_X)\times A^q(X,M,\mathcal{W}_X)\to A^{p+q}(X,M,-\TT_{\Delta}+\VV_X+\mathcal{W}_X)$.

\end{center}

\end{Par}
The preceding considerations and the functoriality of the Gysin maps prove the following theorem.
\begin{The}

If $M$ is a MW-cycle module with a ring structure and $X$ a smooth scheme over $k$, the intersection product turns 
\begin{center}
$\bigoplus_{\VV_X\in \mathfrak{V}(X)}A^*(X,M,\VV_X)$
\end{center}
 into a graded commutative associative algebra over 
\begin{center}
$\bigoplus_{\VV_X\in \mathfrak{V}(X)}A^*(X,\KMW,\VV_X)$.
\end{center}
\end{The}

In particular, we obtain a product on the Chow-Witt ring $\CHt(X)$ which coincides with the intersection product (defined in \cite[§3.4]{Fasel18bis}, see also \cite{Fasel13}). Indeed, our construction of Gysin morphisms follows the classical one (using deformation to the normal cone) and our cross products correspond to the one already defined for the Milnor-Witt K-theory (see \cite[§3]{Fasel18bis}).

\subsection{Homotopy modules}

\begin{Par}
We denote by $\SH(k)$ the stable homotopy category over $k$. It is equipped with the \textit{homotopy} t-structure given by the full subcategory $\SH_{\geq 0}(k)$ (resp. $\SH_{\leq 0}(k)$ consisting of $\PP^1$-spectra $M$ with
\begin{center}
$ \pi_n(M)_m=0$
\end{center}
for each $m\in \ZZ$ and $n<0$ (resp. $n>0$) (see \cite[§5.2]{Mor03}).
\par The heart $\SH(k)^{\heartsuit}$ of this t-structure is equivalent to the category of homotopy modules which, by definition, are Nisnevich sheaves from the category of smooth schemes over $k$ to the category of $\ZZ$-graded abelian groups satisfying the $\AAA^1$-homotopy invariance property. The main theorem of \cite{Fel19} is the following:

\end{Par}
\begin{The} \label{Thesis_Theorem}
The category of Milnor-Witt cycle  modules over $k$ (denoted by $\MW(k)$ is equivalent to the category of homotopy modules (or, equivalently, the heart of the stable homotopy category equipped with its homotopy t-structure):
\begin{center}
$\MW(k) \simeq \SH(k)^{\heartsuit}$.
\end{center}
\end{The}

%
%
%

  \bibliographystyle{alpha}
  \bibliography{exemple_biblio.bib}


\end{document}